\documentclass[reqno, 11pt]{amsart}%
\usepackage{amssymb}
\usepackage[nesting]{hyperref}
\usepackage[pdftex]{graphicx}
\usepackage{listings}
\usepackage{multirow}
\usepackage{placeins}
\usepackage{color}
\usepackage{subfigure}
\usepackage{lscape}
\usepackage{dsfont}
\usepackage{amsmath}
\usepackage{amsfonts}%
\usepackage{tikz}
\usepackage{verbatim}
\providecommand{\U}[1]{\protect\rule{.1in}{.1in}}
\newcommand{\BlackBoxes}{\global\overfullrule5pt}

\BlackBoxes

\newcommand{\R}{\mathbb{R}}
\newcommand{\N}{\mathbb{N}}

\newcommand{\Eop}{\mathbb{E}}
\newcommand{\Pop}{\mathbb{P}}

\newtheorem{theorem}{Theorem}

\newtheorem{lemma}[theorem]{Lemma}

\theoremstyle{definition}

\newtheorem{remark}[theorem]{Remark}
\newtheorem{definition}[theorem]{Definition}
\numberwithin{equation}{section} \numberwithin{theorem}{section}
\def\0{\kern0pt\-\nobreak\hskip0pt\relax}

\makeatletter
\AtBeginDocument{ \def\@serieslogo{ \vbox to\headheight{ \parindent\z@ \fontsize{6}{7\p@}\selectfont
\vss}}}

\def\makeoverbar#1#2#3#4#5#6#7{ \setbox0=\hbox{$\m@th#2\mkern#5mu{{}#3{}}\mkern#6mu$} \setbox1=\null \dimen@=#4\fontdimen8#13 \dimen@=3.5\dimen@
\advance\dimen@ by \ht0 \dimen@=-#7\dimen@ \advance\dimen@ by \wd0
\ht1=\ht0 \dp1=\dp0 \wd1=\dimen@
\dimen@=\fontdimen8#13 \fontdimen8#13=#4\fontdimen8#13
\rlap{\hbox to \wd0{$\m@th\hss#2{\overline{\box1}}\mkern#5mu$}}
\fontdimen8#13=\dimen@}
\def\mylabel#1#2{{\def\@currentlabel{#2}\label{#1}}}
\makeatother

\begin{document}
\title[Markov Decision Processes under Ambiguity ]{Markov Decision Processes under Ambiguity}
\author[N. \smash{B\"auerle}]{Nicole B\"auerle${}^*$}
\address[N. B\"auerle]{Institute for Stochastics,
Karlsruhe Institute of Technology (KIT), D-76128 Karlsruhe, Germany}

\email{\href{mailto:nicole.baeuerle@kit.edu}{nicole.baeuerle@kit.edu}}

\author[U. \smash{Rieder}]{Ulrich Rieder$^\ddag$}
\address[U. Rieder]{Institute of Optimization and Operations Research, University of Ulm,   D-89069 Ulm, Germany}

\email{\href{mailto:ulrich.rieder@uni-ulm.de} {ulrich.rieder@uni-ulm.de}}

\thanks{${}^*$ Institute of Stochastics,
Karlsruhe Institute of Technology (KIT), D-76128 Karlsruhe, Germany}
\thanks{${}^\dagger$ Institute of Optimization and Operations Research,  University of Ulm, D-89069 Ulm, Germany}

\begin{abstract}
We consider  statistical Markov Decision Processes where the decision maker is risk averse against model ambiguity. The latter is given by an unknown parameter which influences the transition law and the cost functions. Risk aversion is either measured by the entropic risk measure or by the Average Value at Risk. We show how to solve these kind of problems using a general minimax theorem. Under some continuity and compactness assumptions we prove the existence of an optimal (deterministic) policy and discuss its computation. We illustrate our results using an example from statistical decision theory.

\end{abstract}
\maketitle


\makeatletter \providecommand\@dotsep{5} \makeatother



\vspace{0.5cm}
\begin{minipage}{14cm}
{\small
\begin{description}
\item[\rm \textsc{ Key words} ]
{\small Model Ambiguity; Bayesian MDP; Entropic Risk Measure; Average Value at Risk; Dual Representation; Minimax Theorem}
\end{description}
}
\end{minipage}

\section{Introduction}
The following experiment has (in a variant) been  suggested by Ellsberg (1961) (see e.g. \cite{gs}): An agent has to choose between two bets. For this she is shown two urns, each containing 100 balls which are either red or black. Urn A contains 50 red and 50 black balls while there is no further information about urn B. Bet 1 is: 'the ball drawn from urn A is black' and bet 2 is: 'the ball drawn from urn B is black'. In case of winning the bet, the agent receives 100 euros. Empirically it has been observed that most agents prefer bet 1. One explanation for this behavior is that  in case of urn B agents consider a set of possible distributions for the colours of the balls and being ambiguity averse take into account the minimal expected utility. 

This point of view has become popular in economics and has been formalized later on. In particular one has to specify the possible set of distributions which have to be taken into account. For example \cite{hs01} consider in the framework of a continuous-time consumption problem the set of distributions $\Pop$ whose relative entropy with respect to a fixed distribution $\Pop_0$ is less or equal to a constant. Using a Lagrange approach this is equivalent to penalizing the robust problem with the distance of the distribution to $\Pop_0.$ Optimization criteria like this have been put on an axiomatic basis by \cite{mmr}.

As far as Markov Decision Processes (MDPs)  are concerned, robust approaches have been considered in \cite{I05} among others. As in \cite{hs01} model ambiguity is here treated with respect to the whole probability measure of the process. Since the probability measure in MDP theory is a product of transition kernels such robust optimization problems can also be interpreted as games against nature.

In this paper we will take another point of view which is related to the models introduced in \cite{kmm}. There, the whole risk is separated into two parts: Model ambiguity and operating risk. One has to operate a system under an unknown probability law which is  chosen by nature (from a finite set) in a worst case way. This model ambiguity is incorporated in the optimization criterion in a risk-sensitive way.  For further literature on ambiguity see the survey \cite{gr}.

We will start with a statistical  Markov Decision Process where the transition kernel and cost functions depend on an unknown parameter for which we have a prior distribution. Only the states of the process are observable. Since the Ellsberg experiment suggests that the parameter (model) ambiguity should be treated different to uncertainty of the evolution of the process, we will consider the expected cost of a policy as a random variable and use the entropic risk measure for the model ambiguity. This is in this specific setting similar to the approach in \cite{kmm}, but different to approaches where the entropic risk measure is applied to the product measure of parameter uncertainty and uncertainty of process evolution. The latter approach has been pursued in \cite{br17} and extended to robust problems in \cite{rs18}. Using the dual representation of the entropic risk measure we can show that there is a connection of our risk-sensitive optimization criterion to the robust penalty problem considered in \cite{hs01}. Relations like this have already been discussed in \cite{o12}. However, note that in our setting nature chooses only the worst case measure with respect to the parameter uncertainty. Our model includes both the classical Bayesian MDP (with risk neutral attitude towards ambiguity) and the robust MDP as limiting cases. We use the general minimax theorem of Kneser, Fan, Sion (see \cite{sion}) and results of \cite{schal} to solve our problem. It is easy to see from our approach that the solution method not only applies to the case where model ambiguity is evaluated by the entropic risk measure, but to any convex risk measure with a suitable dual representation. Thus, we will also consider the case where   model ambiguity is evaluated by the Average Value at Risk. This complements studies in which the Average Value at Risk is applied to the whole discounted cost (see e.g. \cite{bo11, ctmp}).

Our paper is organized as follows: In the next section we introduce our statistical MDP with a given prior distribution $\mu_0$ and the optimization criterion which we consider. Alternative representations and interpretations are also discussed. Section \ref{sec:policies} is then devoted to the minimax theorem and the existence of optimal polices. It will turn out that under some continuity and compactness assumptions, optimal policies exist and coincide with the optimal policy of a classical Bayesian  MDP with different (more pessimistic) prior $\mu^*$ instead of $\mu_0$.  The model with Average Value at Risk is discussed in Section \ref{sec:avar} and yields from a structural point of view the same policy. Section \ref{sec:BayesMDP} explains how the problem can be solved in an algorithmic way. In the last section we explain our approach using a specific example from statistical decision theory. In this example we are able to derive the optimal policy for the entropic risk measure as well as for the Average Value at Risk.

\section{MDP with Entropic Risk Measure for Model Ambiguity}
We suppose that a {\em statistical Markov Decision Processes} is given which we introduce as follows: We assume that  the state space  $E$ is a Borel space, i.e., a Borel subsets of some Polish space endowed with the $\sigma$-algebra of Borel sets.  Actions can be taken from a set $A$ which is again a Borel space. The set $D_n\subset E \times A$ is a Borel subset for $n\in\N_0$. By $D_n(x) := \{a\in A : (x,a)\in D_n\}$ we denote the feasible actions depending on the state  $x$ at time $n$. We assume that $D_n$ contains the graph of a measurable mapping from $E$ to $A$. Furthermore there is a non-empty parameter space $\Theta$ endowed with some $\sigma$-algebra. The  stochastic transition kernel $Q_n^\vartheta$ from $D_{n-1}$ to $E$ which determines the distribution of the new state at time $n$ given the current state and action depends on a parameter $\vartheta\in \Theta$. So $Q_n^\vartheta(B|x,a)$ is the probability  that the next state at time $n$ is in $B\in\mathcal{B}(E)$, given the current state is $x$ and action $a\in D_{n-1}(x)$ is taken. $Q_0^\vartheta$ is the distribution of the initial state.  In what follows we assume that the law of motion is given by 
\begin{eqnarray*}
Q_0^\vartheta (dx) &:=& q_0^\vartheta(x) \lambda_0(dx),\\
Q^\vartheta_n(dx'| x,a) &:=& q_n^\vartheta(x,a,x') \lambda_n(dx').
\end{eqnarray*}
We assume that $\lambda_n$ are probability measures on $E$. Moreover, let
\begin{eqnarray*}
(\vartheta,x) & \mapsto & q_0^\vartheta(x),\\
(\vartheta,x,a,x') &\mapsto & q_n^\vartheta(x,a,x')
\end{eqnarray*} 
be non-negative measurable functions on $\Theta\times E$ and  $\Theta\times D_{n-1}\times E$ for all $n\in\N$ respectively.

\begin{remark}
In general, $\lambda_n$ are assumed to be $\sigma$-finite measures for all $n$. But then there exists a probability measure $\lambda_n^*$ and a finite positive density $f_n(x')$ such that $\lambda_n(dx')=f_n(x') \lambda_n(dx')$. Then we can replace $\lambda_n$ by $\lambda_n^*$ and $q_n^\vartheta(x,a,x')$ by $q_n^\vartheta(x,a,x')f_n(x')$ and w.l.o.g.\ we may assume that $\lambda_n$ are probability measures.
\end{remark}


Next we introduce policies for the decision maker. Here it is important to consider the {\em set of observable histories} which are defined as follows:
\begin{eqnarray*}
  H_0 &:=& E \\
  H_n &:=& \{ (h_{n-1},a_{n-1}) : h_{n-1}\in H_{n-1}, a_{n-1}\in D_{n-1}(x_{n-1}) \} \times E.
\end{eqnarray*}
An element $h_n=(x_0,a_0,x_1,\ldots,x_n)\in H_n$ denotes the observable history of the process up to time $n$ which consists of the sequence of states and actions. For a Borel set $M$ we denote by $\mathcal{P}(M)$ the set of all probability measures on $M$. In what follows we consider MDPs with finite horizon $N\in\N$.

\begin{definition}
\begin{enumerate}
\item[a)] A measurable mapping $\pi_n: H_n\to \mathcal{P}(A)$ with the property that it holds $\pi_n(h_n)( D_n(x_n))=1$ for $h_n\in H_n$ is called a {\em randomized decision rule} at stage $n$.
\item[b)] A sequence $\pi=(\pi_0,\pi_1,\ldots,\pi_{N-1})$ where $\pi_n$ is a randomized decision rule at stage $n$  for all $n$, is called {\em policy}. We denote by $\Pi$ the set of all policies.
 \item[c)] A decision rule $\pi_n: H_n\to \mathcal{P}(A)$ is called {\em deterministic} if  $\pi_n(h_n)= \delta_{f_n(h_n)}$ for some measurable function $f_n:H_n\to A$ with $f_n(h_n)\in D_n(x_n).$ Here $\delta_x$ is the one-point measure on $x$. A policy is called {\em deterministic} if all decision rules are deterministic.
\end{enumerate}
\end{definition}

A policy $\pi=(\pi_0,\pi_1,\ldots,\pi_{N-1})$ induces according to the theorem of Ionescu-Tulcea a probability measure
$$ \Pop_\pi^\vartheta := Q_0 ^\vartheta\otimes \pi_0 \otimes Q_1^\vartheta \otimes \pi_1 \otimes Q_2^\vartheta\ldots\otimes \pi_{N-1}\otimes Q_N^\vartheta$$
on $H_N$. Since $Q_n^\vartheta$ depends measurably on $\vartheta$, we may infer that for any $\pi\in \Pi$, the mapping $(\vartheta,B) \mapsto \Pop_\pi^\vartheta(B)$ is a transition probability from $\Theta$ into $H_N$.

The corresponding stochastic decision process is given by $(X_0,A_0,X_1, A_1,\ldots,X_N)$ and determines the state-action process.

Further we have  measurable and bounded  cost functions $$(\vartheta,x,a) \mapsto c_n^\vartheta(x,a)$$ on $ \Theta\times D_n$ and a measurable and bounded terminal cost 
$$(\vartheta,x)\mapsto g_N^\vartheta(x)$$
on $\Theta\times E$. All cost functions may 
depend on the unknown parameter $\vartheta$. 
Note that in this case we assume that cost are not observable.

We are now interested in the cost incurred by this decision process over the finite time horizon $N$. Therefore, we define for a policy $\pi$ 
$$ C_{N\pi}(\vartheta) := \Eop^\vartheta_\pi\Big[\sum_{n=0}^{N-1}  c_n^\vartheta(X_n,A_n)+ g_N^\vartheta(X_N)\Big]$$
where $ \Eop^\vartheta_\pi$ is the expectation with respect to $ \Pop^\vartheta_\pi$. Note that $\vartheta\mapsto C_{N\pi}(\vartheta)$ is measurable on $\Theta$.
Suppose $\mu_0\in\mathcal{P}(\Theta)$ is a fixed initial belief about the unknown parameter $\vartheta$. In the established theory of Bayesian MDP (see e.g. \cite{br11}, Sec. 5) the aim would be to minimize
\begin{equation}\label{eq:Bayes}
 \int C_{N\pi}(\vartheta)  \mu_0(d\vartheta)  \end{equation}
over all policies $\pi$. This criterion implies that the decision maker is risk neutral with respect to the operating risk as well as with respect to model ambiguity, given in form of the prior $\mu_0$. In what follows we will now consider the case that the decision maker is risk averse with respect to model ambiguity. More precisely, we consider 
\begin{eqnarray} \label{eq:vnpi}
V_N(\pi) &:=& \frac1\gamma \ln\Big(\int e^{\gamma C_{N\pi}(\vartheta) }\mu_0(d\vartheta) \Big),\\ \label{opt:prob}
V_N &:=& \inf_\pi V_N(\pi)
\end{eqnarray} 
with $\gamma>0$. For small $\gamma$ the criterion is approximately equal to (see \cite{bp})
$$V_N(\pi) \approx\int C_{N\pi}(\vartheta) \mu_0(d\vartheta) + \frac12 \gamma  Var_{\mu_0}[C_{N\pi} ].$$
In particular for $\gamma\downarrow 0$ we obtain in the limit the classical Bayesian MDP \eqref{eq:Bayes}. For $\gamma >0$ the variability of the minimal cost in $\vartheta$ is penalized. Moreover, we have the following  representation for \eqref{eq:vnpi}, also known as 'dual' representation (see \cite{fs}, p.279) where
$$ V_N(\pi)= \sup_{\mu\in \mathcal{P}(\Theta)} \Big\{\int C_{N\pi}(\vartheta) \mu(d\vartheta)-\frac1\gamma I(\mu \| \mu_0)\Big\}  $$
where for  $\mu,\nu\in \mathcal{P}(\Theta),  $
$$ I(\mu\| \nu) := \left\{ \begin{array}{cc}
\int\ln(\frac{d \mu}{d \nu})d \mu, & \mbox{ if } \mu\ll \nu,\\
\infty, & \mbox{ else}
\end{array}  \right. $$  is the relative entropy function or Kullback-Leibler distance. From this representation we see that the  case $\gamma\uparrow \infty$ corresponds to the case of a robust optimization problem or worst-case optimization problem where we minimize the cost if nature choose the least favourable measure for the parameter $\vartheta$. For $\gamma>0$ this means that potentially a whole range of beliefs about $\vartheta$ is considered but deviations from the belief $\mu_0$ are penalized.  A similar criterion has been used in \cite{hs01} where preferences  of an agent for a bet $X$ are expressed by
$$ \sup_{\Pop\in \mathcal{P}(\Omega)} \Big\{ \Eop_\Pop[X] -\frac1\gamma I(\Pop \| \Pop_0)\Big\}.$$ 
In our paper we relate model ambiguity only to the unknown parameter $\vartheta$.  This is connected to what in economics is called two-stage approach and where ambiguity typically arises in the first (model) stage. Empirically this has been discovered in the famous Ellsberg experiment. In \cite{kmm} it has been suggested to consider
$$\sum_j p_j \xi\Big( \Eop_{\Pop_j}[U(X)] \Big)$$ 
as a preference function where $\xi$ is an increasing real-valued function which describes the agent's attitude towards ambiguity and $U$ is a utility function.

In what follows we denote by 
$$C_{N\pi}(\mu) :=\int C_{N\pi}(\vartheta) \mu(d\vartheta).$$
When we insert the dual representation in \eqref{opt:prob}, then we obtain \begin{equation}\label{eq:dual}
V_N=\inf_\pi\sup_{\mu\in \mathcal{P}(\Theta)} \Big\{C_{N\pi}(\mu)-\frac1\gamma I(\mu \| \mu_0)\Big\}.\end{equation}
Though it is well-known how the solution of the inner optimization looks like, namely
$$  \hat{\mu}(d \vartheta) := \frac{e^{C_{N\pi}(\vartheta)}}{\int e^{C_{N\pi}(\vartheta')}\mu_0(d\vartheta') } \mu_0(d\vartheta),$$ 
it is impossible to solve the outer minimization problem directly, nor get some information about the structure of the optimal policy, since $\hat{\mu}$ depends on the policy, too.

\section{Existence of Optimal Policies} \label{sec:policies}

It would  be easier to solve the problem if we could interchange the sup and the inf in \eqref{eq:dual}. In order to achieve this we use the general minimax theorem of Kneser, Fan, Sion (see \cite{sion}). The theorem uses the definition of concave-convexlike functions.

\begin{definition}
A function $h$ on $M\times N$ is called concave-convexlike, if
\begin{itemize}
\item[(i)] for all $x_1,x_2\in M$ and $0\le \alpha\le 1$, there is an $x\in M$ such that
$$ \alpha h(x_1,y) +(1-\alpha) h(x_2,y) \le h(x,y),\quad \mbox{ for all } y\in N.$$
\item[(ii)]for all $y_1,y_2\in N$ and $0\le \alpha\le 1$, there is an $y\in N$ such that
$$ \alpha h(x,y_1) +(1-\alpha) h(x,y_2) \ge h(x,y),\quad \mbox{ for all } x\in M.$$
\end{itemize}
\end{definition}

Note in particular that any function $h$ on $M\times N$ which is concave in the first component and convex in the second component is concave-convexlike. Then Theorem 4.2. in \cite{sion} tells us:

\begin{theorem}\label{theo:minimax}
Let $M$  be  any space and $N$ be a compact space,  $h$ a function on $M\times N$ that is concave-convexlike. If  $h(x,y)$ is lower semi-continuous in $y$ for all $x\in M$ then $$\sup_x\min_y h(x,y)= \min_y \sup_x h(x,y).$$
\end{theorem}

We would like to apply the theorem to the function $$L_N(\mu,\pi):= C_{N\pi}(\mu)-\frac1\gamma I(\mu \| \mu_0)$$ which is defined on $\mathcal{P}(\Theta)\times \Pi.$ 
A topology on $\Pi$ can be introduced as follows: Denote by  $\Pop^\lambda_\pi$  the probability measure on $H_N$ defined by
$$ \Pop^\lambda_\pi := \lambda_0 \otimes \pi_0\otimes \lambda_1\otimes \ldots \otimes\pi_{N-1} \otimes \lambda_N$$
and let $\Pi^\lambda :=\{\Pop^\lambda_\pi : \pi\in \Pi \}\subset \mathcal{P}(H_N)$  the set of all probability measures $\Pop^\lambda_\pi$ which are generated by policies.
On $\mathcal{P}(H_N) $ we consider the ws$^\infty$-topology (see \cite{schal75}), i.e. the coarsest topology such that $\Pop \mapsto \int g d\Pop$ is continuous for all $g:H_N\to \R$ such that $(a_0,\ldots,a_{N-1}) \mapsto g(x_0,\ldots ,x_{N}; a_0,\ldots,a_{N-1})$ is continuous and the function $g$  is bounded and measurable.  Given the relativization of the ws$^\infty$-topology on $\Pi^\lambda,$ we can then endow $\Pi$ with the inverse image under the mapping $\pi\mapsto \Pop_\pi^\lambda$ of the topology on $\Pi^\lambda$. This is the coarsest topology on $\Pi$ for which $\pi\mapsto \Pop_\pi^\lambda$ is continuous.

For the next statements we need some assumptions for all $n=0,1,\ldots,N-1$. 

\begin{description}
\item[(C1)] The set $D_n(x)$ is  compact for all $x$.
\item[(C2)] The function $a\mapsto q_n^\vartheta (x,a,x')$ is lower semi-continuous for all $x,x'\in E$ and $\vartheta\in\Theta$.
\item[(C3)] The function $a\mapsto c_n^\vartheta (x,a)$ is lower semi-continuous for all $x\in E$ and $\vartheta\in\Theta$.
\end{description}

Then we obtain:

\begin{lemma}\label{lem:CC}
Under (C1)-(C3) it holds:
\begin{itemize}
\item[a)] $\Pi$ is compact.
\item[b)] The mapping $\pi \mapsto L_N(\mu,\pi)$ is lower semi-continuous on $\Pi$ for all $\mu\in\mathcal{P}(\Theta)$ and for all $\pi_1,\pi_2\in\Pi$ and $\alpha\in(0,1)$ there exists a policy $\pi\in\Pi$ such that $L_N(\mu, \pi)= \alpha L_N(\mu,{\pi_1})+(1-\alpha) L_N(\mu, {\pi_2})$ for all $\mu\in\mathcal{P}(\Theta).$
\item[c)] $\mathcal{P}(\Theta)$ is  convex and $\mu \mapsto L_N(\mu,\pi)$ is concave on $\mathcal{P}(\Theta)$.
\end{itemize} 
\end{lemma}

\begin{proof}
\begin{itemize}
\item[a)] This is Corollary 7.3 b) in \cite{schal}.
\item[b)] Follows from Corollary 7.3 a) in \cite{schal}. It suffices to show that $C_{N\pi}(\mu)$ is lower semi-continuous on $\Pi$. In order to see how our assumptions are needed we give the following sketch of the proof. First note that
\begin{eqnarray*}
C_{N\pi} (\vartheta) &=& \int \Big( \sum_{n=0}^{N-1} c_n^\vartheta(X_n,A_n) + g_N^\vartheta(X_N)\Big) d\Pop_\pi^\vartheta \\
&=& \int \Big( \sum_{n=0}^{N-1} \tilde{c}_n^\vartheta(X_0,A_0,\ldots,X_n,A_n) + \tilde{g}_N^\vartheta(X_0,A_0,\ldots,X_N) \Big)d\Pop_\pi^\lambda\\
&=:& \tilde{C}_N(\Pop_\pi^\lambda,\vartheta)
\end{eqnarray*}
where 
\begin{eqnarray*}
\tilde{c}_n^\vartheta(h_n,a_n) &:=& c_n^\vartheta(x_n,a_n) q_n^\vartheta(x_{n-1},a_{n-1},x_n) \cdots q_0^\vartheta(x_0)\\
 \tilde{g}_N^\vartheta(h_N) &:=& g_N^\vartheta(x_N) q_{N-1}^\vartheta(x_{N-1},a_{N-1},x_N) \cdots q_0^\vartheta(x_0).
\end{eqnarray*}
We obtain $\tilde{C}_N(\Pop_\pi^\lambda,\mu) := \int \tilde{C}_N(\Pop_\pi^\lambda,\vartheta)\mu(d\vartheta) =C_{N\pi}(\mu)$.
Then one can show (using the assumptions) that $\Pop\mapsto \tilde{C}_N(\Pop,\mu)$ is lower semi-continuous on $\Pi^\lambda$ in the ws$^\infty$-topology. The lower semi-continuity of $C_{N\pi} (\mu) $ on $\Pi$ follows since $\pi \mapsto\Pop_\pi^\lambda$ is continuous.

Note that for the second statement it is important to work with randomized policies.
\item[c)] The convexity of the set is obvious.  Concavity of the mapping can also be shown: For this purpose let $\mu_i\ll \mu_0, i=1,2.$ According to the Radon-Nikodym theorem $\mu_i$ have densities w.r.t.\ $\mu_0$ say $\mu_i = \int g_i d\mu_0$. Hence for $\alpha\in(0,1)$ the measure $\mu := \alpha\mu_1+(1-\alpha)\mu_2$ has density $\alpha g_1+(1-\alpha)g_2$ w.r.t. $\mu_0$ and we consider
$$  I(\mu \| \mu_0) = \int \ln\big(\alpha g_1(\vartheta)+(1-\alpha)g_2(\vartheta)\big)(\alpha g_1(\vartheta)+(1-\alpha)g_2(\vartheta))d\vartheta.$$
Since obviously $x\mapsto x\ln(x)$ is convex for $x>0$,  we obtain that $\mu \mapsto -\frac{1}{\gamma} I(\mu\|\mu_0)$ is concave. Since $\mu\mapsto C_{N\pi}(\mu)$ is linear, the  statement follows. 
\end{itemize}
\end{proof}

All assumptions of Theorem \ref{theo:minimax} are now satisfied and we obtain

\begin{theorem}\label{theo:Minimax1}
Under assumptions (C1)-(C3) it holds:
\begin{itemize}
\item[a)] Min and sup can be interchanged, i.e. \begin{equation}
\min_\pi\sup_{\mu\in \mathcal{P}(\Theta)} L_N(\mu,\pi) = V_N= \sup_{\mu\in \mathcal{P}(\Theta)} \min_\pi L_N(\mu,\pi).
\end{equation}
\item[b)] There exists an optimal policy $\pi^*$  for \eqref{opt:prob}, i.e. $V_N(\pi^*) =V_N.$ 
\end{itemize}
\end{theorem}

\begin{proof}
Part a) is a direct consequence of Theorem \ref{theo:minimax} and the fact that $L_N(\mu,\pi)$ is lower semi-continuous in $\pi$ due to Lemma \ref{lem:CC} on the compact set $\Pi$.
Part b) follows from a) since $\pi\mapsto \sup_{\mu\in \mathcal{P}(\Theta)}L_N(\mu,\pi)$ is lower semi-continuous. 
\end{proof}

For the second main theorem we need further conditions for all $n=0,1,\ldots, N-1$:
\begin{description}
\item[(C4)] The parameter space $\Theta$ is a  compact metric space (endowed with the $\sigma$-algebra of Borel subsets of $\Theta$).
\item[(C5)] The function $(\vartheta,a)\mapsto q_n^\vartheta(x,a,x')$ is lower semi-continuous on $\Theta\times D_{n-1}(x)$ for all $x,x'\in E$.
\item[(C6)] The function $(\vartheta,a)\mapsto c_n^\vartheta (x,a)$ is continuous on $\Theta\times D_{n-1}(x)$ for all $x\in E$ and the function $\vartheta \mapsto g_N^\vartheta(x)$ is continuous for all $x\in E$.
\end{description}

 These assumptions imply that we obtain a worst prior measure (initial belief). Here we endow $\mathcal{P}(\Theta)$ with the weak topology.

\begin{theorem}\label{theo:Minimax2}
In addition to (C1),(C2) assume that (C4)-(C6) are satisfied.  Then it holds:
\begin{itemize}
\item[a)] There exists a saddle point  $(\mu^*,\pi^*)$ of the function $(\mu,\pi)\mapsto L_N(\mu,\pi)$ and  \begin{equation}
\min_\pi\max_{\mu\in \mathcal{P}(\Theta)} L_N(\mu,\pi) = L_N(\mu^*,\pi^*)=V_N=\max_{\mu\in \mathcal{P}(\Theta)} \min_\pi L_N(\mu,\pi).
\end{equation}
\item[b)] The policy $\pi^*$ is an optimal policy for \eqref{opt:prob} and $\pi^*$ is an optimal Bayes policy w.r.t.\ $\mu^*$, i.e. $C_{N\pi^*}(\mu^*)= \inf_\pi C_{N\pi}(\mu^*).$
\item[c)]  There exists a deterministic policy $f^*:=(f_0^*,\ldots, f_{N-1}^*)$
 with $C_{Nf^*}(\mu^*)=C_{N\pi^*}(\mu^*)$, i.e. $f^*$ is optimal for \eqref{opt:prob}.
\end{itemize}
\end{theorem}

\begin{proof}
a) $\Theta$ compact implies that $\mathcal{P}(\Theta)$ is weakly compact.  Moreover, the mapping $\mu \mapsto L_N(\mu,\pi)$ is  upper semi-continuous in the weak topology, since $\mu \mapsto \int \Eop^\vartheta_\pi[C_{N\pi}(\vartheta)] \mu(d\vartheta)$ is continuous by our assumptions  (see Corollary 8.3 in \cite{schal})  and the entropy function $\mu \mapsto I(\mu \| \mu_0)$ is lower semi-continuous w.r.t.\  the weak topology (see Theorem 1 in \cite{posner}).  Hence also $\mu\mapsto \inf_\pi L_N(\mu,\pi)$ is upper semi-continuous and attains its supremum on $\mathcal{P}(\Theta)$. The existence of a saddle point $(\mu^*,\pi^*)$ follows from the classical saddle point theorem of game theory. 

Part b) follows directly from a) since $V_{N}(\pi^*)=V_N$ is equivalent to $C_{N\pi^*}(\mu^*)=\inf_\pi C_{N\pi}(\mu^*)$. Part c) is well-known in Bayesian MDPs and follows with \cite{hin}, Theorem 15.2 together with Lemma 15.1.
\end{proof}

Theorem \ref{theo:Minimax2} has the advantage that it is possible to solve the inner optimization problem $\min_\pi L_N(\mu,\pi)$ explicitly. Since the entropy part does not depend on the policy $\pi$, only the part $C_{N\pi}(\mu)$ is interesting and it can be solved with the established theory of Bayesian MDP (see Section \ref{sec:BayesMDP}). Of course, the resulting optimal policy depends on $\mu^*$ which has to be computed in a second step. 

\begin{remark}
It is possible to consider MDPs with infinite time horizon in the same way. I.e.\ instead of $C_{N\pi}(\vartheta)$ we take 
$$ C_{\infty\pi}(\vartheta) := \Eop^\vartheta_\pi\Big[\sum_{n=0}^{\infty}  c^\vartheta_n(X_n,A_n)\Big]$$
and assume $\sum_{n=0}^\infty \sup_{\vartheta,x,a}| c_n^\vartheta(x,a)|<\infty$ or a weaker convergence assumption. Then we obtain the same results as for finite-stage MDPs with agents who are ambiguity averse.

\end{remark}

\begin{remark}
The entropic risk measure motivates to penalize the robust MDP formulation by the deviation of the prior from the 'statistically correct' prior $\mu_0$. Instead of taking the relative entropy one could of course take any other distance which is convex. For example one could take the {\em Bhattacharyya distance} which for probability measures $\mu$ and $\mu_0$ with densities $\varphi$ and $\varphi_0$ is defined by
$$ D_B(\mu,\mu_0) := -\log\Big(\int \sqrt{\varphi(x)\varphi_0(x)}dx\Big)$$  and is a convex mapping in $\mu$ for fixed $\mu_0$. For details see \cite{bat}.
\end{remark}

\section{MDP with Average Value at Risk for Model Ambiguity}\label{sec:avar}
Instead of the entropic risk measure one may apply any other convex risk measure to penalize model ambiguity. Convex risk measures have a representation in dual form (see \cite{fs}, Theorem 4.33) which can be used to apply the minimax Theorem.  In what follows we restrict the discussion to the Average Value at Risk. 
We consider the same Bayesian MDP framework as in Section 2 with a fixed initial belief $\mu_0$ and define the Value at Risk at level $\gamma\in(0,1)$ for the random variable $\vartheta\mapsto C_{N\pi}(\vartheta)$  on $\Theta$  as $$VaR_\gamma(C_{N\pi}) := \inf\{x\in\R : \mu_0(C_{N\pi}\le x) \ge \gamma\}.$$
Note that we consider the actuarial point of view here where large positive outcomes are bad and $\gamma$ is usually close to 1. Moreover, note that $VaR_\gamma(C_{N\pi}) $ depends on $\mu_0$.

 When model ambiguity is measured by the Average Value at Risk, we obtain as optimization criterion
\begin{eqnarray}
 V_N(\pi) &:=&  \frac{1}{1-\gamma} \int_\gamma^1 VaR_\alpha(C_{N\pi})d\alpha.\\
  V_N &:=& \inf_\pi V_N(\pi). \label{eq:avarprob}
\end{eqnarray}
If $\gamma \downarrow 0$ we get in the limit just the expectation and thus the classical risk neutral setting. For $\gamma\uparrow 1$  we obtain in the limit the worst-case risk measure. 
Using the dual representation of Average Value at Risk (see e.g. \cite{fs}, Theorem 4.52) this amounts to
$$ V_N = \inf_\pi \sup_{\mu\in\mathcal{Q}_\gamma} C_{N\pi}(\mu)$$
with $C_{N\pi}(\mu) := \int C_{N\pi}(\vartheta)\mu(d\vartheta)$ and $$\mathcal{Q}_\gamma := \Big\{ \mu\in \mathcal{P}(\Theta) : \mu \ll \mu_0, \frac{d\mu}{d\mu_0}\le \frac{1}{1-\gamma}\Big\}.$$ The idea here is to proceed in the same way as in Section 3 and use the previously established results. We obtain with some slight changes to the previous section:
\newpage 

\begin{theorem}\label{theo:CC2}
Under  (C1),(C2) and (C4)-(C6) it holds:
\begin{itemize}
\item[a)] There exists a saddle point  $(\mu^*,\pi^*)$ of the function $(\mu,\pi)\mapsto C_{N\pi}(\mu)$ and  \begin{equation}
\min_\pi\max_{\mu\in \mathcal{Q}_\gamma}C_{N\pi}(\mu) = C_{N\pi^*}(\mu^*)=V_N=\max_{\mu\in \mathcal{Q}_\gamma} \min_\pi C_{N\pi}(\mu).
\end{equation}
\item[b)] The policy $\pi^*$ is an optimal policy for \eqref{eq:avarprob}, i.e. $V_N(\pi^*)=V_N$, and $\pi^*$ is an optimal Bayes policy w.r.t.\ $\mu^*$, i.e. $C_{N\pi^*}(\mu^*)= \inf_\pi C_{N\pi}(\mu^*).$
\item[c)]  There exists a deterministic policy $f^*:=(f_0^*,\ldots,f_{N-1}^*)$  with $C_{Nf^*}(\mu^*)=C_{N\pi^*}(\mu^*)$ and $f^*$ is optimal for \eqref{eq:avarprob}, i.e. $V_N(f^*)=V_N$.
\end{itemize}
\end{theorem}

\begin{proof}
First note that Lemma \ref{lem:CC} holds in the same way since $\mathcal{Q}_\gamma$ is convex. The statements follow as in the proof of Theorem \ref{theo:Minimax2} since $\mathcal{Q}_\gamma$ is weakly compact  (see Corollary 4.38 in \cite{fs}).
\end{proof}

\section{Solving the Bayesian Dynamic Decision Problem}\label{sec:BayesMDP}
Let $\mu_0\in\mathcal{P}(\Theta)$ be fixed. The problem $\inf_\pi C_{N\pi}(\mu_0)=C_N(\mu_0)$ is a standard Bayesian MDP. For a detailed explanation how these problems can be solved, see \cite{br11}, Section 5.   We give here only a rough idea. Note that in the context of Section 3 and 4 we have to replace $\mu_0$ by $\mu^*$ in the solution procedure.  
The problem can be solved by a state space augmentation. The state which has to be considered is $(x,\mu)$ where $x\in E$ and $\mu\in \mathcal{P}(\Theta)$ is the current belief (conditional distribution) about $\vartheta$. This belief has to be updated as follows:
\begin{eqnarray*}
\mu_0(x)(C) &:=&  \frac{\int_C q^\vartheta_0(x)\mu_0(d\vartheta)}{\int_\Theta q^\vartheta_0(x)\mu_0(d\vartheta)}, \quad C\in \mathcal{B}(\Theta),\\
 \Phi_n(x,\mu,a,x')(C) &:=& \frac{\int_C q_n^\vartheta(x,a,x')\mu(d\vartheta)}{\int_\Theta q_n^\vartheta(x,a,x')\mu(d\vartheta)}, \quad C\in \mathcal{B}(\Theta).
\end{eqnarray*}

$\Phi_n(x,\mu,a,x') $ gives the new belief, if the previous belief was $\mu$, the previous state was $x$, the new state is $x'$ and action $a$ is chosen. $\mu_0(x)$ is the new belief directly after the observation of the first state.  Thus, starting with the prior $\mu_0$ we obtain a sequence of beliefs $\mu_n(h_{n-1},a_{n-1},x_n):=  \Phi_n(x_{n-1},\mu_{n-1}(h_{n-1}),a_{n-1},x_n)$ depending on the observations and the history of the process: $\mu_0(x_0),\mu_1(x_0,a_0,x_1), \mu_2(h_2),\ldots$. The state transition kernel  is given by
$$Q^X_n(B|x,\mu,a) = \int Q_n^\vartheta(B|x,a) \mu(d\vartheta), \quad B \in \mathcal{B}(E). $$
Under well-known continuity and compactness assumptions it is then possible to show that the value $$C_N(\mu_0)=\int  \int J_0(x,\mu_0(x)) Q_0^\vartheta(dx) \mu_0(d \vartheta)$$ 
can be computed recursively by
\begin{eqnarray*}
J_N(x,\mu) &:=& \int g^\vartheta_N(x)\mu(d\vartheta),\\
J_n(x,\mu) &:=& \inf_{a\in D_n(x)} \Big\{ \int_\Theta c_n^\vartheta(x,a) \mu(d\vartheta) +  \int J_{n+1} (x',\Phi_{n+1}(x,\mu,a,x')) Q_{n+1}^X(dx'|x,\mu,a)\Big\}.
\end{eqnarray*}
If we denote by $g_n^*$ the (deterministic) minimizer of $J_{n+1}$ on the right-hand side of the equation for $n=0,1,\ldots, N-1$, then the deterministic policy $\pi^*:=(f^*_0,\ldots,f^*_{N-1})$ is optimal for the given problem with
$$ f^*_n(h_n):= g_{n}^*(x_n, \mu_n(h_n)),\quad h_n\in H_n, \quad n=0,1,\ldots,N-1$$ i.e. $C_{N\pi^*}(\mu_0)= C_N(\mu_0)$.

\section{An Example}
We consider the following example which is taken from \cite{dG}, Example 2, Section 12.6. A statistician  observes a sequence of Bernoulli random variables with unknown success probability $\vartheta$. Suppose that $\vartheta$ is either $\frac13$ or $\frac23$. She has an initial belief $\mu_0$ about the two probabilities. Two actions are available: Either stop the observation process and choose a terminal decision or make a further observation of the Bernoulli trial. The cost of one observation is $1$ and if the decision is correct, there is no cost. For a wrong terminal decision one has to pay the amount of $10$. What is the optimal Bayesian strategy? We assume that the statistician is risk averse and uses the criteria presented in this paper.  

We use Theorem \ref{theo:Minimax2} resp.\ Theorem \ref{theo:CC2} to solve these problems. 
We have to take as the state space the current belief about the two hypothesis. Since the parameter set is $\Theta=\{ \frac13, \frac23 \}$ these beliefs are only two-point distributions. In what follows we assume that the interval $[0,1]$ is the state space where $\mu\in[0,1]$ is the current belief that $\vartheta=\frac13$ is the true parameter. The action space is $A=\{1,2\}$ where $a=1$ means to take another observation and $a=2$ means to stop the observation process and choose a terminal  decision (which is then the hypothesis with higher belief). 
In this case the cost is given by
$$ c(\mu) := \min\{10\mu,10(1-\mu)\}.$$
In case we decide to take another observation and the observation is a 'success' (indicated by '1') we obtain the following new belief:
$$\Phi(\mu,1)= \frac{\frac13\mu}{\frac13\mu+\frac23(1-\mu)}=\frac{\mu}{2-\mu}.$$
In case we observe a 'failure' (indicated by '0') we obtain for the new belief
$$\Phi(\mu,0)= \frac{\frac23\mu}{\frac23\mu+\frac13(1-\mu)}=\frac{2\mu}{1+\mu}.$$
Then we obtain from the  Bayesian MDP theory the following recursion:
\begin{eqnarray*}
C_0(\mu) &=& c(\mu),\\
C_n(\mu) &=& \min\Big\{ c(\mu) ; 1+ \Big(\frac13 \mu +\frac23 (1-\mu)\Big) C_{n-1}\Big(\frac{\mu}{2-\mu}\Big)+\Big(1-\frac13 \mu -\frac23 (1-\mu)\Big) C_{n-1}\Big(\frac{2\mu}{1-\mu}\Big)\Big\}.
\end{eqnarray*}  

Working through this recursion finally yields:
$$C_0(\mu) = \left\{ \begin{array}{cl}
10 \mu, & 0\le \mu\le \frac12,\\
10(1-\mu), & \frac12 < \mu\le 1.
\end{array}\right.$$
and $C_n=C_1$ for all $n\in\N$, where
$$C_1(\mu) = \left\{ \begin{array}{cl}
10 \mu, & 0\le \mu\le \frac{13}{30},\\[0.1cm]
\frac{13}{3}, & \frac{13}{30}< \mu\le \frac{17}{30},\\[0.1cm]
10(1-\mu), & \frac{17}{30} < \mu\le 1.
\end{array}\right.$$
The optimal decision at the beginning is to take another observation $(a=1)$ if $\mu\in (\frac{13}{30}, \frac{17}{30}),$ otherwise take a  terminal decision $(a=2)$ immediately. After one observation the statistician will always take a terminal decision.

\subsection{Problem with Entropic Risk Measure}
When we want to solve the  problem with risk aversion against ambiguity measured by the entropic risk measure, we now have to consider the following problem where $\mu_0\in (0,1)$ is the initial belief
$$ V_N= \sup_{0<\mu<1} \Big\{C_1(\mu)-\frac1\gamma I(\mu\| \mu_0)\Big\}.$$ 
Using the fact that the function $C_1$ is symmetric i.e. $C_1(\mu)=C_1(1-\mu)$ this boils down to 
\begin{eqnarray}
V_N&=& \max\Big\{   \sup_{0<\mu<\frac{13}{30}} \Big\{10 \mu -\frac1\gamma \Big(\mu \ln( \frac{\mu}{\mu_0})+(1-\mu) \ln(\frac{1-\mu}{1-\mu_0}) \Big);\\
&&  \hspace*{1.2cm}\sup_{\frac{13}{30}<\mu\le \frac12} \Big\{\frac{13}{3} -\frac1\gamma \Big(\mu \ln( \frac{\mu}{\mu_0})+(1-\mu) \ln(\frac{1-\mu}{1-\mu_0}) \Big) \Big\}.
\end{eqnarray}
When the statistician is risk averse with parameter $\gamma=0.1$ and has initial belief $\mu_0=0.1$ about the hypothesis $\vartheta=\frac13$ she will rather solve the Bayesian MDP with $\mu^*=0.232$. Observe that $\mu=\frac12$ is most risk averse choice of the prior. In Figure \ref{fig:erm} the optimal $\mu^*$ is plotted as a function of the risk aversion $\gamma$ for different $\mu_0$. Note that
\begin{eqnarray*}
V_N &=& C_1(\mu^*) -\frac1\gamma I(\mu^*\| \mu_0) \mbox{ if } \gamma >0\\
V_N &=&  C_1(\mu_0) \mbox{ if } \gamma =0.
\end{eqnarray*}

\begin{figure}[htb]
\includegraphics[scale=0.6]{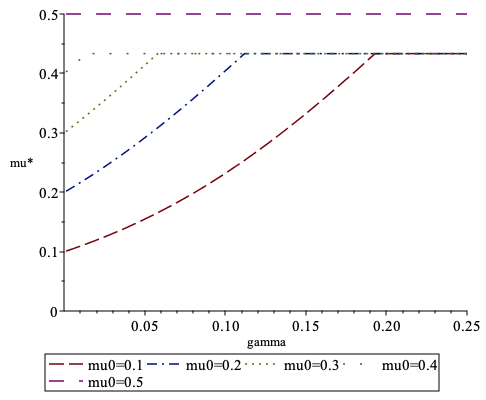}
\caption{$\mu^*$ as a function of $\gamma$ for different $\mu_0$.}\label{fig:erm}
\end{figure}

What we observe in the example is that
\begin{itemize}
\item[(i)] For $\mu_0\in [0,\frac12]$ we have $\mu_0\le \mu^*\le \frac12$.
\item[(ii)] $\lim_{\gamma\downarrow 0} \mu^*(\gamma) = \mu_0.$
\item[(iii)] $\lim_{\gamma\uparrow \infty} \mu^*(\gamma)  \in [\frac{13}{30},\frac12].$
\end{itemize}

The interpretation of (i) is that a risk averse statistician will always shift the statistically correct prior in direction of the uniform distribution. The case   $\lim_{\gamma\downarrow 0} $ is in the limit the classical Bayesian MDP with original prior $\mu_0$. The case $\lim_{\gamma\uparrow \infty}$ corresponds to the robust optimization where the most unfavourable prior is chosen. In this example the most unfavourable prior is any prior in the interval $[\frac{13}{30},\frac12]$ since this requires another observation.

\subsection{Problem with Average Value at Risk}
We can also consider this example with the ambiguity measured by the Average Value at Risk.  Here  we have to solve
$$ \max_{\mu\in \mathcal{M}_\gamma} C_1(\mu), \quad \mathcal{M}_\gamma := \Big\{\mu \in (0,1) : \frac{\mu}{\mu_0}\le \frac{1}{1-{\gamma}}, \frac{1-\mu}{1-\mu_0}\le \frac{1}{1-{\gamma}}\Big\}.$$
The set $\mathcal{M}_\gamma$ corresponds to $\mathcal{Q}_\gamma$.
 The maximum point $\mu^* $ as a function of $\gamma$ and $\mu_0\le \frac12$ is given by (in case on non-uniqueness we give the whole rang of optimal values)
$$ \mu^* = \left\{ \begin{array}{cl}
\frac{\mu_0}{1-\gamma}, & \gamma\le 1-\frac{30}{13}\mu_0\\[0.1cm]
(\frac{13}{30},\frac{\mu_0}{1-\gamma} ), & \gamma \in (1-\frac{30}{13}\mu_0,1-\frac{30}{17}\mu_0)\\[0.1cm]
(\frac{13}{30},\frac{17}{30} ), & \gamma > 1-\frac{30}{17}\mu_0.
\end{array}\right.$$
Due to symmetry reasons we restrict again to the case $\mu_0\le \frac12$. In Figure \ref{fig:avar} the optimal $\mu^*$ is plotted as a function of $\gamma$ for different $\mu_0$. Note that 
\begin{eqnarray*}
V_N &=& C_1(\mu^*) \mbox{ if } \gamma >0\\
V_N &=&  C_1(\mu_0)\mbox{ if } \gamma =0.
\end{eqnarray*}

\begin{figure}[htb]
\includegraphics[scale=0.7]{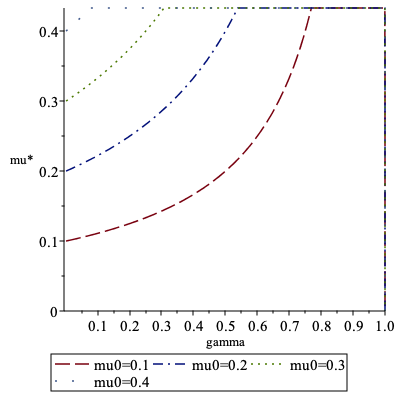}
\caption{$\mu^*$ as a function of $\gamma$ for different $\mu_0$.}\label{fig:avar}
\end{figure}

We again observe in this case that
\begin{itemize}
\item[(i)] For $\mu_0\in [0,\frac12]$ we have $\mu_0\le \mu^*\le \frac12$.
\item[(ii)] $\lim_{\gamma\downarrow 0} \mu^*(\gamma) = \mu_0.$
\item[(iii)] $\lim_{\gamma\uparrow 1} \mu^*(\gamma)  \in [\frac{13}{30},\frac12].$
\end{itemize}

Though in this case the interpretation of $\gamma$ is different, the general behaviour of the optimal $\mu^*$ is the same.

\section{Conclusion}
In this paper we present a proposal to deal with  model ambiguity for MDPs. Using a dual representation and a general minimax theorem we are able to solve the ambiguity problem. The solution procedure is illustrated by an example taken from statistical decision theory. The approach is closely related to robust MDPs.

\bibliographystyle{plainnat}

\end{document}